\documentclass{birkjour}
\makeatletter
\RenewDocumentCommand\ref{s}{%
	\IfBooleanTF{#1}{\@kernel@sref}{\@kernel@ref}%
}
\RenewDocumentCommand\pageref{s}{%
	\IfBooleanTF{#1}{\@kernel@spageref}{\@kernel@pageref}%
}
\makeatother

\usepackage{amssymb}
\usepackage{amsmath}

\usepackage{enumitem}

\newtheorem{thm}{Theorem}[section]

\newtheorem{lem}{Lemma}[section]
\newtheorem{prop}{Proposition}[section]
\theoremstyle{definition}
\newtheorem{defn}{Definition}[section]
\theoremstyle{remark}
\newtheorem{rem}{Remark}
\newtheorem{exam}{Example}[section]
\numberwithin{equation}{section}

\usepackage[backend=bibtex,style=numeric]{biblatex}
\begin{document}
	
	\title[Spectral theorem for compact self-adjoint linear relations]
	{Spectral theorem for compact self-adjoint linear relations with an application to elliptic boundary value problems}
	
	\author[MATHEW THOMAS] {MATHEW THOMAS$^*$}
	\address{Research Scholar\\
		Department of Mathematics\\
		Cochin University of Science and Technology\\
		Kochi, Kerala, India.\\
		https://orcid.org/0000-0001-9708-763X}
	\email{mathewmylettu@cusat.ac.in}
	
	\thanks{$^*$Corresponding author\\
		Email address: mathewmylettu@cusat.ac.in (\textsc{MATHEW THOMAS})}
	
	\author[V. B. KIRAN KUMAR]{V. B. KIRAN KUMAR}
	\address{Assistant Professor\\
		Department of Mathematics\\
		Cochin University of Science and Technology\\
		Kochi, Kerala, India.\\
		https://orcid.org/0000-0001-7643-4436}
	\email{vbk@cusat.ac.in}
	
	
	
	\date{\today}
	
	\begin{abstract}
		We establish a spectral theorem for compact self-adjoint linear relations on a Hilbert space. The proof relies on the Arens decomposition of a closed linear relation into an operator part and a purely multivalued part, reducing the problem directly to the classical spectral theorem for compact self-adjoint operators. No finite-dimensionality condition is imposed on the multivalued part. We also prove a finite-rank approximation theorem for compact bounded linear relations and a Fredholm alternative for compact self-adjoint linear relations. As an application, we employ the theory of boundary triples to show that the inverse of the Neumann Laplacian on a bounded domain is a compact self-adjoint linear relation. The abstract spectral theorem then recovers the classical eigenfunction expansion of solutions to the Neumann problem: an $L^2$-convergent series in terms of the non-constant Neumann eigenfunctions together with an arbitrary additive constant. The framework extends naturally to general elliptic boundary value problems with finite-dimensional kernel.
	\end{abstract}
	
	\maketitle
	
	\section{Introduction}
	
	The study of multivalued linear operators, or linear relations, is a dynamic and expanding field within functional analysis. Extending classical operator theory to linear relations has attracted significant interest since the foundational work of J.~von~Neumann \cite{0d,0e}, who introduced linear relations in 1932 while examining adjoints of non-densely defined differential operators. When an operator is not densely defined, its adjoint is naturally multivalued; this is the primary motivation for studying linear relations.
	
	In 1961, Arens \cite{21} introduced the adjoint of a linear relation without requiring dense domain and, crucially, decomposed any closed Hermitian linear relation into an orthogonal sum of an operator part and a purely multivalued part. This decomposition, which we recall in Proposition~\ref{arens} below, provides the essential bridge between linear relations and classical operators, allowing one to transfer many operator-theoretic results to the multivalued setting.
	
	The spectral theorem for compact self-adjoint operators is a cornerstone of functional analysis, with applications ranging from integral equations to Sturm--Liouville problems. In this article, we extend this theorem to compact self-adjoint linear relations on a Hilbert space. A key feature of our result is that it holds without any restriction on the dimension of the multivalued part $T(0)$. The proof works directly with the operator part $T_s$ via the Arens decomposition of a closed linear relation, thereby reducing the problem to the classical spectral theorem for compact self-adjoint operators.
	
	\subsection*{Relation to existing work}
	
	The object studied in this paper a self-adjoint linear relation on a Hilbert space also appears in the literature under the name \emph{self-adjoint subspace}. Shi, Shao, and Ren \cite{4} initiated a systematic study of self-adjoint subspaces, including their resolvent convergence and spectral approximation properties~\cite{3}, and the stability of essential spectra under compact perturbations~\cite{32}. A type of spectral approximation for continuous self-adjoint linear relations was studied in \cite{thomas2026spectral}. That framework works with essentially the same Arens-type decomposition into an operator part and a purely multivalued part on separable Hilbert spaces. While the compact case of the self-adjoint subspace program should be expected to yield an eigenvector expansion in the spirit of Theorem~\ref{thm:spectral}, the explicit orthonormal eigenbasis, the convergent series representation with the additive $T(0)$-term, and crucially the fact that the theorem holds without any finite-dimensionality condition on $T(0)$ do not appear to have been stated in that body of work. The present paper makes these elements explicit and provides a self-contained proof using only the Arens decomposition and the classical spectral theorem for compact self-adjoint operators.
	
	On the Banach space side, Acharya and McBride~\cite{AcharyaMcBride2019} proved that for a compact linear relation $T$ on a Banach space $X$, $\sigma(T) = \sigma_p(T) \cup \{0\}$, using the quotient operator on $X/T(0)$ rather than an orthogonal decomposition. Their result is more general in that it does not require a Hilbert space structure or self-adjointness, but it is weaker in conclusion: it provides neither an orthonormal eigenbasis nor an explicit series representation. The present paper exploits the Hilbert space structure and self-adjointness to make the Arens splitting orthogonal, which yields the stronger expansion.
	
	In view of these comparisons, the primary contributions of this paper are: (i) the explicit spectral theorem with orthonormal eigenbasis and convergent series representation, obtained without any dimension restriction on $T(0)$; (ii) the finite-rank approximation theorem (Theorem \ref{thm:approx}); (iii) the Fredholm alternative for compact self-adjoint linear relations (Theorem \ref{thm:fredholm}); and (iv) the application to the Neumann Laplacian via boundary triples (Section~4).
	
	The second half of the article is devoted to an application of the abstract theory to elliptic boundary value problems. A classical difficulty in the spectral theory of the Neumann Laplacian is that the constant functions form a one-dimensional kernel, so the inverse is not a single-valued operator. Within the framework of boundary triples developed by Derkach and Malamud~\cite{DerkachMalamud}, the inverse of the Neumann Laplacian is naturally identified as a compact self-adjoint linear relation with one-dimensional multivalued part. Applying our spectral theorem to this relation yields a rigorous eigenfunction expansion for solutions of the Neumann problem, recovering the well-known series in terms of the non-constant eigenfunctions plus an arbitrary additive constant.
	
	The article is organized as follows. Section~2 recalls the necessary preliminaries on linear relations, including the Arens decomposition, selections, quotient norms, and spectral definitions. Section~3 contains the main abstract results: finite-rank approximation (Section~3.1)
    and  the Fredholm alternative (Section~3.2). Section~3.3 presents the spectral theorem for compact self-adjoint linear relations.  Section~4 is devoted to the application: we review boundary triples (Section~4.1), construct the Neumann Laplacian and its inverse as a linear relation (Section~4.2), and apply the spectral theorem to obtain the eigenfunction expansion (Section~4.3). Section~5 contains concluding remarks and directions for future work.
	
	\section{Preliminaries on linear relations}
	
	Throughout, $H$ denotes a complex Hilbert space with inner product $\langle\cdot,\cdot\rangle$, which is linear in the first argument. The product space $H \times H$ is endowed with the inner product
	\[
	\langle(x_1,y_1),(x_2,y_2)\rangle = \langle x_1,x_2\rangle + \langle y_1,y_2\rangle,
	\]
	making it a Hilbert space.
	
	\subsection{Basic definitions and operations}
	
	\begin{defn}\label{def:LR}
		A \emph{linear relation} $T$ in $H$ is a linear subspace of $H \times H$. We denote by $LR(H)$ the set of all linear relations on $H$. The domain, range, and multivalued part of $T$ are defined respectively by:
		\begin{align*}
			D(T) &= \{x \in H \mid (x,y) \in T \text{ for some } y \in H\},\\
			R(T) &= \{y \in H \mid (x,y) \in T \text{ for some } x \in H\},\\
			T(0) &= \{y \in H \mid (0,y) \in T\}.
		\end{align*}
		For $x \in D(T)$, we write $T(x) = \{y \in H \mid (x,y) \in T\}$. The graph of $T$ is $G(T) = T$ itself. $T$ is \emph{closed} if $G(T)$ is closed in $H \times H$.
	\end{defn}
	
	A single-valued linear relation is identified with a linear operator in the usual way: if $T(0) = \{0\}$, then $T$ is a linear operator. In this paper, we refer to single-valued linear relations simply as linear operators.
	
	For $T, S \in LR(H)$ and $\alpha \in \mathbb{C}$, the algebraic operations are defined as follows:
	\begin{align*}
		\alpha T &= \{(x, \alpha y) \mid (x,y) \in T\},\\
		T + S &= \{(x, y+z) \mid (x,y) \in T,\; (x,z) \in S\},\\
		T - \alpha I &= \{(x, y - \alpha x) \mid (x,y) \in T\}.
	\end{align*}
	The inverse of $T$ is $T^{-1} = \{(y,x) \mid (x,y) \in T\}$, and the product is
	\[
	ST = \{(x,z) \in H \times H \mid \exists\, y \in H \text{ with } (x,y) \in T,\; (y,z) \in S\}.
	\]
	
	The \emph{adjoint} of $T$ is defined by
	\[
	T^* = \{(f,g) \in H \times H \mid \langle g,x\rangle = \langle f,y\rangle \text{ for all } (x,y) \in T\}.
	\]
	$T^*$ is always a closed linear relation. $T$ is \emph{Hermitian} (or symmetric) if $T \subset T^*$ and \emph{self-adjoint} if $T = T^*$.
	
	\subsection{Compactness and quotient norms}
	
	Let $T \in LR(H)$. The closure $\overline{T(0)}$ is a closed subspace of $H$. Let $Q: H \to H/\overline{T(0)}$ denote the natural quotient map. For each $x \in D(T)$, the set $T(x)$ is an affine subspace parallel to $T(0)$, so its image under $Q$ is a single element of $H/\overline{T(0)}$. Thus the composition
	\[
	QT: D(T) \to H/\overline{T(0)}, \qquad (QT)(x) = Q(T(x))
	\]
	is a single-valued linear operator.
	
	\begin{defn}[\cite{11}]\label{def:compact}
		A linear relation $T \in LR(H)$ is \emph{compact} if $QT$ is a compact operator. That is, for every bounded sequence $(x_n)$ in $D(T)$, the sequence $(QT(x_n))$ has a convergent subsequence in $H/\overline{T(0)}$.
	\end{defn}
	
	\begin{defn}\label{def:norm}
		For $x \in D(T)$, the norm of $T(x)$ is
		\[
		\|T(x)\| = \|(QT)(x)\|_{H/\overline{T(0)}} = \inf_{y \in T(x)} \|y\|.
		\]
		The norm of $T$ is $\|T\| = \sup\{\|T(x)\| : x \in D(T),\, \|x\| \le 1\}$. A linear relation $T$ is \emph{continuous} if $\|T\| < \infty$ and \emph{bounded} if it is continuous with $D(T) = H$.
	\end{defn}
	We can deduce the following Proposition from Proposition 3.1 \cite{11}.
	\begin{prop}\label{prop:compact-seq}
		A linear relation $T \in LR(H)$ is compact if and only if for every bounded sequence $(x_n)$ in $D(T)$, there exist a subsequence $(x_{n_k})$ and a sequence $(y_{n_k})$ with $y_{n_k} \in T(x_{n_k})$ such that $(y_{n_k})$ converges in $H$.
	\end{prop}

\begin{defn}
Let $T \in LR(X)$ and $\lambda \in \mathbb{C}$.  
\begin{enumerate}
    \item The linear relation, $(T - \lambda I)^{-1},$ is called the resolvent of $T$ at $\lambda$.
    \item The resolvent set of $T$ is defined by,
    \[
    \rho(T) = \left\{ \lambda \in \mathbb{C} \ \middle| \ (T - \lambda I)^{-1} 
    \ \text{is a bounded linear operator on } X \right\}.
    \]
  The spectrum, $\sigma(T)$ of $T$ is the complement of the resolvent set $\rho(T)$ in $\mathbb{C}$, and $\sigma(T)$ is always a closed subset of $\mathbb{C}$.

    \item The point spectrum of $T$ is,
    \[
    \sigma_p(T) = \left\{ \lambda \in \mathbb{C} \ \middle| \ (x, \lambda x) \in T \ 
    \text{for some} \ x \in X \ \text{with} \ \|x\| = 1 \right\}.
    \]
   Each \(\lambda \in \sigma_p(T)\) is called an eigenvalue of  \(T\), and the corresponding $x$ is an eigenvector.
    \item The essential spectrum, $\sigma_e(T)$, consists of those points of $\sigma(T)$ that are either
    accumulation points of $\sigma(T)$ or isolated eigenvalues of infinite multiplicity.
\end{enumerate}
\end{defn}

	\subsection{The Arens decomposition}

   We prove Arens decomposition for the closed Hermition case for a linear relation. This is the same obtained in \cite{21}. However, we present the whole proof for the sake of computation. 

	\begin{prop}[Arens decomposition]\label{arens}
		Let $T \in LR(H)$ be closed and Hermitian. Then $T$ admits the orthogonal direct sum decomposition
		\[
		T = T_s \oplus T_\infty,
		\]
		where
		\begin{align*}
			T_\infty &= T \cap (\{0\} \times H) = \{0\} \times T(0),\\
			T_s &= T \cap \bigl(T(0)^\perp \times T(0)^\perp\bigr).
		\end{align*}
		Here $T_s$ is a single-valued linear operator with
		\[
		D(T_s) = D(T) \cap T(0)^\perp, \qquad R(T_s) \subset T(0)^\perp,
		\]
		and $T_\infty$ is the purely multivalued part. The decomposition is orthogonal in $H \times H,$ that is, for any $(x,y) \in T_s$ and $(0,z) \in T_\infty$, we have ${\langle (x,y), (0, z) \rangle}_{H \times H} = 0$.
	\end{prop}
	
	\begin{proof}
		For a Hermitian relation $T \subset T^*$, we have $T(0) \subset T^*(0)$. To see this, let $x\in T(0)$, then $(x,0)\in T$, and hence $(x,0)\in T^*$, which gives $x\in T^*(0)$.
		This inclusion is essential for the orthogonality of the decomposition.

		Define $T_\infty = \{0\} \times T(0)$ and $T_s = T \cap (T(0)^\perp \times T(0)^\perp)$.

		\noindent If $(x,y) \in T_s$ and $(0,z) \in T_\infty$, then $y \in T(0)^\perp$ and $z \in T(0)$, so\\ ${\langle (x,y), (0, z) \rangle}_{H \times H} = 0$. Hence we achieve the orthogonality mentioned in the statement.

		For the decomposition of $T$, consider $(x,y) \in T$. Decompose $y = y_1 + y_2$ with $y_1 \in T(0)^\perp$ and $y_2 \in T(0)$ (closedness of $T(0)$ follows from closedness of $T$). Since $y_2 \in T(0)$, we have $(0, y_2) \in T$. Therefore $(x,y_1)=(x,y)-(0,y_2)\in T$. Moreover, $y_1\in T(0)^\perp$, we have $(x,y_1)\in T\cap(H\times T(0)^\perp)$.\\
        To conclude $(x, y_1) \in T_s$, we need $x \in T(0)^\perp$. For that, let $z \in T(0),$ then $z \in T^*(0)$, that is, $(0, z) \in T^*$   ; since $T$ is Hermitian, for any $(x, y_1) \in T$ , we have $\langle y_1, 0\rangle = \langle x, z\rangle$, that is $ \langle x, z\rangle=0$. This means $x \perp T(0)$, that is, $x \in T(0)^\perp$. Thus $(x, y_1) \in T_s$, and the decomposition $T = T_s \oplus T_\infty$ is proved.
		
		Finally, we show that  $T_s$ is single-valued. If $(x, y_1), (x, y_1') \in T_s$, then $(0, y_1 - y_1') \in T$, so $y_1 - y_1' \in T(0)$. But $y_1, y_1' \in T(0)^\perp$, so $y_1 - y_1' \in T(0) \cap T(0)^\perp = \{0\}$. Thus $y_1 = y_1'$, proving $T_s$ is single-valued.

Also, we have
$
D(T_s)=D(T)\cap T(0)^\perp.
$
To prove this, let $x\in D(T_s)$. Then there exists $y_1\in T(0)^\perp$ such that $(x,y_1)\in T_s$. By the definition of $T_s$, we have $(x,y_1)\in T$, and hence $x\in D(T)$. Also, $x\in T(0)^\perp$. Thus
$
x\in D(T)\cap T(0)^\perp,
$
and consequently
$
D(T_s)\subseteq D(T)\cap T(0)^\perp.
$ Conversely, let $x\in D(T)\cap T(0)^\perp$. Then there exists $y\in H$ such that $(x,y)\in T$. Write
$
y=y_1+y_2,\qquad y_1\in T(0)^\perp,\quad y_2\in T(0).
$
As shown above, $(x,y_1)\in T_s$. Hence $x\in D(T_s)$, and therefore
$
D(T)\cap T(0)^\perp\subseteq D(T_s).
$
Thus,
$
D(T_s)=D(T)\cap T(0)^\perp.
$\\
Furthermore, $R(T_s) \subset T(0)^\perp$ follows directly from the definition of $T_s$.

	\end{proof}

	\begin{rem}\label{rem:arens-general}

        From \cite{21}, we get a general decomposition for closed relations on Hilbert spaces. However the domain $D(T_s)=D(T) \cap T(0)^\perp$ in our presentation above. This will be used later.
	\end{rem}
	\begin{rem}\label{rem:key-observation}
		For a self-adjoint linear relation $T$, we have
		\[
		T(0) = D(T)^\perp,
		\]
		and hence
		\[
		\overline{D(T)} = T(0)^\perp.
		\]
		In particular,
		\[
		D(T) \subset T(0)^\perp
		\quad\text{and}\quad
		D(T_s) = D(T).
		\]
		Moreover,
		\[
		D(T) \cap T(0) = \{0\}.
		\]
	\end{rem}
	\begin{lem}[{\cite{1}, Lemma 2.7, Corollary 3.2}]\label{lem:Ts-norm}
		Let $T \in LR(H)$ be closed and Hermitian.
		\begin{enumerate}[label=(\roman*)]
			\item If $T$ is bounded, then $\|T\| = \|T_s\|$.\label{item:norm}
			\item For any $n \ge 1$, $(T^n)_\infty = T_\infty$ and $T^n = T_s^n \oplus T_\infty$, where $T_s^n = (T_s)^n$.\label{item:power}
		\end{enumerate}
	\end{lem}
	
	\begin{thm}[{\cite{3}, Lemma 2.2, Lemma 2.3}]\label{thm:spectrum-decomp}
		Let $T \in LR(H)$ be closed and Hermitian. Then:
		\begin{enumerate}[label=(\roman*)]
			\item $T_\infty = T \cap T(0)^2$, \ $T_s = T \cap (T(0)^\perp)^2$, and $T_s$ is closed and Hermitian in $(T(0)^\perp)^2$, while $T_\infty$ is closed and Hermitian in $T(0)^2$.
			\item $\rho(T) = \rho(T_s)$, $\sigma(T) = \sigma(T_s)$, $\sigma_p(T) = \sigma_p(T_s)$, $\sigma_e(T) = \sigma_e(T_s)$.\label{item:spectrum-eq}
			\item If $T$ is self-adjoint, then $T_s$ is a self-adjoint operator in the Hilbert space $T(0)^\perp$, and $T_\infty$ is self-adjoint in $T(0)$. Consequently, $\sigma(T) \subset \mathbb{R}$.
		\end{enumerate}
	\end{thm}
    
{ 
\begin{lem}[Converse to Theorem \ref{thm:spectrum-decomp}(iii)]\label{lem:converse-self-adjoint}
	Let $H$ be a Hilbert space, let $V \subseteq H$ be a closed subspace, and let
	$S$ be a self-adjoint operator on $V^\perp$ (that is, $D(S) \subseteq V^\perp$
	is dense in $V^\perp$ and $S = S^*$ computed within $V^\perp$). Define
	\[
	T := \{(x, Sx + v) : x \in D(S),\ v \in V\} \subset H \times H.
	\]
	Then $T$ is a self-adjoint linear relation in $H$, with $D(T) = D(S)$,
	$T(0) = V$, and $T_s = S$.
\end{lem}

\begin{proof}
	$T$ is clearly a linear subspace of $H\times H$, so $T \in LR(H)$.
	
	First we notice that $T(0)=V:$ since $0 \in D(S)$ and $S$ is linear, $S(0)=0$. Hence
	$(0,y) \in T$ if and only if $y = S(0)+v = v$ for some $v \in V$, that is if and only if $y \in V$.
	Thus $T(0)=V$.
	
	Next we see that $T^* = T.$ Let $(f,g) \in T^*$, that is \
	\begin{equation}\label{eq:conv-defn}
		\langle g,x\rangle = \langle f,\,Sx+v\rangle
		\qquad\text{for all } x\in D(S),\ v\in V.
	\end{equation}
	
	Then $f \in V^\perp:$ taking $x=0$ in \eqref{eq:conv-defn} gives
	$0=\langle g,0\rangle=\langle f,v\rangle$ for every $v \in V$.

	Write $g=g_V+g_{V^\perp}$ with $g_V\in V$, $g_{V^\perp}\in V^\perp$ (using
	$H=V\oplus V^\perp$, $V$ closed). Since $f\in V^\perp$,
	$\langle f,v\rangle=0$ for all $v\in V$, so \eqref{eq:conv-defn} reduces to
	\[
	\langle g,x\rangle=\langle f,Sx\rangle
	\qquad\text{for all } x\in D(S).
	\]
	As $D(S)\subseteq V^\perp$, we have $\langle g,x\rangle=\langle g_{V^\perp},x\rangle$,
	so
	\[
	\langle g_{V^\perp},x\rangle=\langle f,Sx\rangle
	\qquad\text{for all } x\in D(S),
	\]
	that is\ $(f,g_{V^\perp})\in S^*$ (adjoint computed within $V^\perp$). Since $S$
	is self-adjoint, $S=S^*$, so $f\in D(S)$ and $g_{V^\perp}=Sf$.
	
	Therefore, for any $f\in D(S)$ and any $g_V\in V$, the pair
	$(f,\,Sf+g_V)$ satisfies \eqref{eq:conv-defn}. Hence for all $x\in D(S)$, $v\in V$,
	\[
	\langle Sf+g_V,\,x\rangle
	=\langle Sf,x\rangle
	=\langle f,Sx\rangle
	=\langle f,Sx+v\rangle,
	\]
	using $x\in D(S)\subseteq V^\perp \Rightarrow \langle g_V,x\rangle=0$,
	self-adjointness of $S$, and $f\in V^\perp \Rightarrow \langle f,v\rangle=0$.
	
	Combining the above calculations, we get,
	\[
	T^* = \{(f,\,Sf+g_V) : f\in D(S),\ g_V\in V\} = T.
	\]
	
	Finally, $D(T)=D(S)$ is immediate from the definition of $T$, and $T_s = T \cap (T(0)^\perp\times T(0)^\perp) = T\cap(V^\perp\times V^\perp)$
	consists exactly of the pairs $(x,Sx)$ with $x\in D(S)$ (taking $v=0$ in the
	definition of $T$, and noting $Sx+v\in V^\perp$ forces $v=0$ since $Sx\in V^\perp$),
	so $T_s=S$.
\end{proof}

}
	\begin{thm}[{\cite{4}, Theorem 2.2}]\label{thm:nullspace}
		Let $T$ be a closed Hermitian linear relation in $H$. Then $N(T - \lambda) = N(T_s - \lambda)$ for every $\lambda \in \sigma_p(T)$.
	\end{thm}
	
	\begin{lem}\label{lem:compact-Ts}
		Let $T \in LR(H)$ be closed and Hermitian.
		\begin{enumerate}[label=(\roman*)]
			\item If $T$ is compact, then $T_s$ is a compact operator on $T(0)^\perp$.
			\item Conversely, if $T$ is bounded and $T_s$ is a compact operator on $T(0)^\perp$, then $T$ is compact.
		\end{enumerate}
	\end{lem}
	
	\begin{proof}
		(i) Suppose $T$ is compact. Let $(x_n)$ be a bounded sequence in $D(T_s) = D(T) \cap T(0)^\perp$. Since $T$ is compact, $QT$ is compact, so there exists a subsequence $(x_{n_k})$ such that $(QT(x_{n_k}))$ converges in $H/\overline{T(0)}$. For $x \in T(0)^\perp$, we have $T(x) = T_s(x) + T(0)$, and therefore $QT(x) = T_s(x) + \overline{T(0)}$. Since the norm in the quotient space satisfies
		\[
		\|QT(x)\|_{H/\overline{T(0)}} = \operatorname{dist}(T_s(x), \overline{T(0)}) = \|T_s(x)\|
		\]
		(the last equality holds because $T_s(x) \in T(0)^\perp$ and the distance from an element of $T(0)^\perp$ to $\overline{T(0)}$ is exactly its norm), the canonical map
		$J:T(0)^\perp\to H/\overline{T(0)}$, $J(z)=z+\overline{T(0)}$, is an isometry. Hence $(T_s(x_{n_k}))$ is Cauchy in $T(0)^\perp$ and since $T(0)^\perp$ is complete, $(T_s(x_{n_k}))$ converges. Thus $T_s$ is a compact operator.
		
		(ii) Suppose $T$ is bounded and $T_s$ is compact. Let $(x_n)$ be a bounded sequence in $D(T)$. Decompose each $x_n = u_n + v_n$ where $u_n = P x_n \in D(T_s)$ and $v_n = (I-P)x_n \in T(0)$, with $P$ the orthogonal projection onto $T(0)^\perp$. Both $(u_n)$ and $(v_n)$ are bounded. By compactness of $T_s$, there is a subsequence $(u_{n_k})$ such that $T_s(u_{n_k}) \to y$ for some $y \in T(0)^\perp$. Then
		\[
		QT(x_{n_k}) = Q(T_s(u_{n_k}) + T(0)) = T_s(u_{n_k}) + \overline{T(0)} \to y + \overline{T(0)}
		\]
		in $H/\overline{T(0)}$. Hence $QT$ is compact, so $T$ is compact.
	\end{proof}
	
	\begin{rem}\label{rem:compact-self-adjoint-operator-part}
		If $T$ is compact and self-adjoint, then the operator part $T_s$ is a closed self-adjoint operator on $T(0)^\perp$ and is compact in the sense of Definition~\ref{def:compact}. The sequential characterization of compactness implies that $T_s$ is bounded: otherwise there would exist a bounded sequence $(x_n)$ in $D(T_s)$ for which $\|T_sx_n\|\to\infty$, contradicting compactness. Since $T_s$ is closed and bounded, its domain is closed; since it is self-adjoint, its domain is dense. Hence
		\[
		D(T_s)=T(0)^\perp,
		\]
		and $T_s$ is an everywhere-defined compact self-adjoint operator on $T(0)^\perp$. Thus the classical compact self-adjoint spectral theorem applies to $T_s$.
	\end{rem}
	
	\subsection{Selections}
	
	\begin{defn}[{\cite{0a}, Definition I.5.1}]\label{def:selection}
		Let $T \in LR(H)$. A linear operator $A$ is called a \emph{selection} of $T$ if
		\[
		T = A + (\{0\} \times T(0)) \quad \text{and} \quad D(T) = D(A).
		\]
		Equivalently, $T(x) = A(x) + T(0)$ for all $x \in D(T)$.
	\end{defn}
	
	\begin{rem}\label{rem:selection-norm}
		If $A$ is a selection of $T$, then for any $x \in D(T)$,
		\[
		\|T(x)\| = \inf_{z \in T(0)} \|A(x) + z\| \le \|A(x)\|,
		\]
		so $\|T\| \le \|A\|$.
	\end{rem}

	\subsection*{Finite-rank linear relations}
	
	\begin{defn}\label{def:finite-rank}
		A linear relation $T \in LR(H)$ is of \emph{finite rank} if $Dim(R(T)/T(0)) < \infty$. Equivalently, the induced operator $QT: D(T) \to H/\overline{T(0)}$ has finite-dimensional range.
	\end{defn}
	\section{Abstract results}
	
	\subsection{Finite-rank approximation of compact linear relations}
	
\begin{lem}\label{compact-selection}
	Let $H$ be a Hilbert space and let $T \in LR(H)$ be compact, bounded, and closed. Then $T$ admits a compact linear selection $A_0: H \to H$.
\end{lem}

\begin{proof}
	Since $T$ is bounded, we have $D(T) = H$. Let $Q: H \to H/\overline{T(0)}$ denote the quotient map.
	
First we show that $T(0)$ is closed.
	
	Indeed, $T(0) = \{y \in H : (0, y) \in T\}$ is the preimage of $T$ under the continuous map $y \mapsto (0, y)$. Since $T$ is closed, its preimage is closed. Alternatively, if $y_n \in T(0)$ and $y_n \to y$, then $(0, y_n) \to (0, y)$ in $H \times H$. Since $T$ is closed, $(0, y) \in T$, so $y \in T(0)$. Therefore, $T(0) = \overline{T(0)}$.

	Since $T$ is compact, $QT: H \to H/T(0)$ is a compact operator (here we use $T(0) = \overline{T(0)}$). Since $T$ is bounded, $QT$ is bounded. Therefore, $QT$ is a compact bounded linear operator from $H$ to the Hilbert space $H/T(0)$.

Let $x+T(0)\in H/T(0)$. Since $T(0)$ is closed, by the projection theorem, for every $x\in H$ there exist unique $x_1\in T(0)^\perp$ and $x_2\in T(0)$ such that
$x=x_1+x_2$. Then, $x+T(0)=x_1+x_2+T(0)=x_1+T(0)$.
\begin{center}
    Define, $R:H/T(0)\longrightarrow H$ by
$R(x+T(0))=x_1$.
\end{center}
Then $(QR)(x+T(0))=Q(x_1)=x_1+T(0)=x+T(0)$. Hence, $QR=I_{H/T(0)}$.

	That is, $R$ be the  right inverse of $Q$ mapping $H/T(0)$ onto $T(0)^\perp$. 
	
	Define:
	\[
	A_0 := R \circ QT.
	\]
	Then $A_0: H \to H$ is compact (composition of compact $QT$ with bounded $R$) and linear.
	
	Now we show that $A_0$ is a selection of $T$. Let $x \in H$, since $D(T) = H$, there exists $y \in T(x)$. We have:
	\[
	Qy = QTx = Q(R(QTx)) = QA_0x.
	\]
	Therefore, $y - A_0x \in T(0)$ (since $Q(y - A_0x) = Qy - QA_0x = 0$ in $H/T(0)$).
	
	Since $(x, y) \in T$ and $(0, y - A_0x) \in T$ (because $y - A_0x \in T(0)$), the linearity of $T$ gives:
	\[
	(x, A_0x) = (x, y) - (0, y - A_0x) \in T.
	\]
	Therefore, $A_0x \in T(x)$ for all $x \in H$, proving that $A_0$ is a selection of $T$.
	
	finally we show that $A_0$ is compact.
	
	Since $A_0 = R \circ QT$ where $R$ is bounded and $QT$ is compact, $A_0$ is compact.
	
	Therefore, $A_0: H \to H$ is a compact linear selection of $T$.
\end{proof}
\begin{lem}\label{lem:finite-rank-construction}
	Let $T \in LR(H)$ be as in Lemma~\ref{compact-selection}, let $A$ be a compact
	selection of $T$, and let $B$ be a bounded finite-rank operator on $H$. Then
	each of the linear relations
	\[
	BA + T(0), \qquad AB + T(0), \qquad BAB + T(0)
	\]
	is of finite rank.
\end{lem}

\begin{proof}
	Since $B$ has finite rank, so do $BA$ (as $R(BA)\subseteq R(B)$), $AB$ (as
	$R(AB)=A(R(B))$, the image of a finite-dimensional space under a linear map,
	and $BAB$ (as $R(BAB)\subseteq R(B)$). Write $C$ for any one of $BA$, $AB$,
	$BAB$, and let $k={Dim} ( R(C))<\infty$. For any $(x,y)\in C+T(0)$,
	we have $y=C(x)+z$ for some $z\in T(0)$. In the quotient, the induced
	operator sends $x\mapsto Q(C(x))$, so its range is contained in $Q(R(C))$,
	which has dimension at most $k$. Hence $C+T(0)$ is of finite rank by
	Definition~\ref{def:finite-rank}.
\end{proof}
	
	\begin{lem}\label{lem:operator-approx}
		Let $A$ be a compact operator on a Hilbert space $H$, and let $(P_n)$ be a sequence of finite-rank orthogonal projections such that $P_n \to I$ strongly. Then
		\[
		\|P_n A - A\| \to 0, \qquad \|A P_n - A\| \to 0, \qquad \|P_n A P_n - A\| \to 0
		\]
		as $n \to \infty$.
	\end{lem}

	\begin{thm}[Finite-rank approximation]\label{thm:approx}
		Let $H$ be a Hilbert space and let $T \in LR(H)$ be compact, bounded, and closed. Let $A: H \to H$ be a compact selection of $T$ (which exists by Lemma~\ref{compact-selection}). Let $(P_n)$ be a sequence of finite-rank orthogonal projections on $H$ such that $P_n \to I$ strongly. Define the linear relations
		\[
		T_n^P = P_n A + T(0), \qquad
		T_n^S = A P_n + T(0), \qquad
		T_n^G = P_n A P_n + T(0).
		\]
		Then each $T_n^P, T_n^S, T_n^G$ is a finite-rank linear relation, and
		\[
		\|T_n^P - T\| \to 0, \qquad \|T_n^S - T\| \to 0, \qquad \|T_n^G - T\| \to 0
		\]
		as $n \to \infty$.
	\end{thm}\begin{proof}
By Lemma~\ref{lem:finite-rank-construction}, applied with $B=P_n$, each of
$T_n^P=P_nA+T(0)$, $T_n^S=AP_n+T(0)$, and $T_n^G=P_nAP_n+T(0)$ is a finite-rank
linear relation.
		
		Now we prove the norm convergence for $T_n^P$. We have
		\[
		T_n^P - T = (P_n A + T(0)) - (A + T(0)).
		\]
		For $x \in H$,
		\[
		(T_n^P - T)(x) = P_n A(x) - A(x) + T(0).
		\]
		Thus $P_n A - A$ is a selection of $T_n^P - T$, and by Remark~\ref{rem:selection-norm},
		\[
		\|T_n^P - T\| \le \|P_n A - A\| \to 0
		\]
		by Lemma~\ref{lem:operator-approx}.\\
        
\noindent Similar proof will work for the convergence of  $T_n^P. $\\

\noindent For the convergence of $T_n^G$, take $x \in H$,
		\[
		(T_n^G - T)(x) = P_n A P_n(x) - A(x) + T(0) = (P_n A P_n - A)(x) + T(0),
		\]
		so $P_n A P_n - A$ is a selection of $T_n^G - T$.
		Since
		\[
		P_nAP_n - A = P_n(AP_n - A) + (P_nA - A),
		\]
		we have
		\[
		\|P_nAP_n - A\| \le \|P_n\|\|AP_n - A\| + \|P_nA - A\| = \|AP_n - A\| + \|P_nA - A\| \to 0.
		\]
		By Remark~\ref{rem:selection-norm},
		\[
		\|T_n^G - T\| \le \|P_nAP_n - A\| \to 0.
		\]
	\end{proof}
	
	\begin{exam}\label{ex:shift}
		Let $H = \ell^2(\mathbb{N})$ and define $F \in LR(H)$ by
		\[
		F(x_1, x_2, x_3, \dots) = \{(c, x_1, \tfrac{x_2}{2}, \tfrac{x_3}{3}, \dots) \mid c \in \mathbb{C}\}.
		\]
		Here $F(0) = \operatorname{span}\{e_1\}$. Write $F = K + F(0)$ where
		\[
		K(x_1, x_2, x_3, \dots) = (0, x_1, \tfrac{x_2}{2}, \tfrac{x_3}{3}, \dots).
		\]
		$K$ is a compact operator, and $K$ is a selection of $F$. Thus $F$ is compact and closed.
		
		Let $P_n$ be the orthogonal projection onto the first $n$ coordinates. Define $F_n^P = P_n K + F(0)$. The quotient norm gives
		\[
		\|F_n^P - F\|^2 \leq \sup_{\|x\|=1} \sum_{k=n}^{\infty} \bigl|\tfrac{x_k}{k}\bigr|^2 \le \frac{1}{n^2} \to 0.
		\]
		Thus $\|F_n^P - F\| \to 0$.\\
        Similar computations will work for $F_n^S$ and $F_n^G.$
	\end{exam}
		\subsection{The Fredholm alternative}
	
	\begin{thm}[Fredholm Alternative for Compact Selfadjoint Linear Relations]\label{thm:fredholm}
		Let $H$ be a Hilbert space and let $T \in LR(H)$ be compact and self-adjoint. Then:
		\begin{enumerate}[label=(\alph*)]
			\item Exactly one of the following holds:
			\begin{enumerate}[label=(\roman*)]
				\item For every $y \in H$, $y \in (I - T)(x)$ has a solution $x \in D(T)$.
				\item There exists nonzero $x \in D(T)$ with $x \in T(x)$ (that is, $1 \in \sigma_p(T)$).
			\end{enumerate}
			Moreover, if (ii) holds, then $Dim ( N(I - T)) < \infty$.

			\item $y \in (I - T)(x)$ has a solution for every $y$ if and only if  $0 \in (I - T)(x)$ has only the zero solution.
		\end{enumerate}
	\end{thm}
	\begin{proof}
		Since $T$ is self-adjoint, $T(0) = T^*(0)$ automatically, and by Proposition~\ref{arens} (in its self-adjoint form), $T$ decomposes as $T = T_s \oplus T_\infty$, where $T_s$ is a single-valued compact self-adjoint operator on the Hilbert space $K := T(0)^\perp$, with $D(T_s) = D(T) \cap T(0)^\perp$, and $T_\infty = \{0\} \times T(0)$. By Remark~\ref{rem:key-observation}, $D(T) \subset T(0)^\perp$, so in fact $D(T) = D(T_s)$.
		
		\textbf{(a)} For $x \in D(T) = D(T_s)$, we have $T(x) = T_s(x) + T(0)$, so
		\[
		(I-T)(x) = x - T_s(x) - T(0) = (I - T_s)(x) - T(0).
		\]
		Hence $y \in (I-T)(x)$ if and only if $(I-T_s)(x) = y + z_0$ for some $z_0 \in T(0)$, that is, if and only if $y \in (I-T_s)(x) - T(0)$.
		
		By the classical Fredholm alternative for the compact self-adjoint operator $T_s$ on $K$, exactly one of the following holds:
		\begin{enumerate}
			\item[(i$_s$)] $I - T_s$ is a bijection of $K$ onto $K$;
			\item[(ii$_s$)] $N(I-T_s) \neq \{0\}$ in $K$.
		\end{enumerate}
		
		Suppose (i$_s$) holds. Given any $y \in H$, decompose $y = y_K + y_0$ with $y_K \in K$, $y_0 \in T(0)$. Since $I - T_s$ is onto $K$, there exists $x \in D(T_s)$ with $(I-T_s)(x) = y_K$. Then
		\[
		y = y_K + y_0 = (I-T_s)(x) + y_0 \in (I-T_s)(x) - T(0) = (I-T)(x),
		\]
		since $-y_0 \in T(0)$. This proves (a)(i).\\ For uniqueness modulo $T(0)$: if $x_1, x_2 \in D(T_s)$ both satisfy $y \in (I-T)(x_j)$, then $(I-T_s)(x_1) - (I-T_s)(x_2) \in T(0)$. But $(I-T_s)(x_1) - (I-T_s)(x_2) \in K$ (as $T_s$ maps into $K$ and $x_1, x_2 \in K$), and $K \cap T(0) = \{0\}$, so $(I-T_s)(x_1 - x_2) = 0$. Since $I - T_s$ is injective on $K$ (from (i$_s$)), then $x_1 = x_2$; in particular the solution is unique in $D(T_s)$.
		
		Suppose instead (ii$_s$) holds: there is a nonzero $u \in D(T_s) \subset D(T)$ with $T_s(u) = u$. Then $u \in T_s(u) + T(0) = T(u)$, so $u \in T(u)$, proving (a)(ii). Moreover,
		\[
		N(I-T) = \{x \in D(T) \mid 0 \in (I-T)(x)\} = \{x \in D(T_s) \mid (I-T_s)(x) \in T(0)\}.
		\]
		Since $(I-T_s)(x) \in K$ for $x \in D(T_s)$ and $K \cap T(0) = \{0\}$, this forces $(I-T_s)(x) = 0$, that is $N(I-T) = N(I-T_s)$. By the classical Fredholm alternative for the compact operator $T_s$, $Dim( N(I-T_s)) < \infty$, so $Dim (N(I-T) )< \infty$.
		
		Finally, (i) and (ii) cannot both hold, since (i$_s$) and (ii$_s$) are mutually exclusive and each of (a)(i), (a)(ii) forces the corresponding one of (i$_s$), (ii$_s$).

		\textbf{(b)} By part (a), the inhomogeneous inclusion $y \in (I-T)(x)$ has a solution for every $y \in H$ if and only if alternative (a)(i) holds, that is\ if and only if (i$_s$) holds, that is\ if and only if (ii$_s$) fails, that is\ if and only if $N(I-T_s) = \{0\}$. By the identity $N(I-T) = N(I-T_s)$ from part (a), this is equivalent to $N(I-T) = \{0\}$, that is\ to the homogeneous inclusion $0 \in (I-T)(x)$ having only the zero solution (indeed, since $N(I-T) \subset D(T_s) \subset T(0)^\perp$ meets $T(0)$ only at $0$).
	\end{proof}
   \begin{rem}
	Unlike in the operator case, the dimensions of $N(I-T)$ and $N(I-T^*)$
	need not agree for an arbitrary compact linear relation $T$. The following
	example illustrates this.
\end{rem}
    \begin{exam}
        
Let $T\in LR(\ell^2)$ with $D(T)=\ell^2$ be defined by
\[T(x_1,x_2,x_3,\ldots)
=
\left\{
\left(c,x_1,\frac{x_2}{2^2},\frac{x_3}{3^2},\ldots\right)
\,\middle|\, c\in\mathbb{C}
\right\},
\qquad x=(x_1,x_2,x_3,\ldots)\in D(T).\]

\noindent Clearly, $T$ is compact. By the definition of the adjoint linear relation, we obtain
\[T^*(0,x_1,x_2,x_3,\ldots)
=
\left(x_1,\frac{x_2}{2^2},\frac{x_3}{3^2},\ldots\right),
\qquad D(T^*)=\{e_1\}^{\perp}.\]

\noindent Consequently, for $x=(0,x_1,x_2,\ldots)\in D(T^*)$,
\[(I-T^*)(x)
=
\left(
-x_1,\,
x_1-\frac{x_2}{2^2},\,
x_2-\frac{x_3}{3^2},\,
x_3-\frac{x_4}{4^2},\ldots
\right).\]

 \noindent Therefore, the inclusion $0\in (I-T^*)(x)$ implies successively
we get, $x=0$. Thus, $x=0$ is the only solution of the inclusion $0\in (I-T^*)(x)$.

\noindent On the other hand,
\[(I-T)(x)
=
\left\{
\left(
x_1-c,\,
x_2-x_1,\,
x_3-\frac{x_2}{2^2},\,
x_4-\frac{x_3}{3^2},\ldots
\right)
\,\middle|\,c\in\mathbb{C}
\right\}.\]

\noindent Hence, the inclusion $0\in (I-T)(x)$ is equivalent to
$x_1=c,\qquad x_n=\frac{x_{n-1}}{(n-1)^2},\qquad n\geq2.$ It follows that
$x=
\left(
c,c,\frac{c}{(2!)^2},\frac{c}{(3!)^2},\ldots
\right),
\qquad c\in\mathbb{C}.$

\noindent Therefore,
$N(I-T)
=
\left\{
\left(
c,c,\frac{c}{(2!)^2},\frac{c}{(3!)^2},\ldots
\right)
\,\middle|\,c\in\mathbb{C}
\right\}
\neq\{0\}.$

    \end{exam}

	\begin{exam}\label{ex:fredholm1}
		Let $T \in LR(\ell^2)$ with $D(T) = e_1^\perp$ be defined by,
		\[
		T(x_1, x_2, x_3, \dots) = \{(c, x_2, \tfrac{x_3}{2^2}, \tfrac{x_4}{3^2}, \dots) \mid c \in \mathbb{C}\}, \qquad x \in D(T).
		\]
		Here $T(0) = \operatorname{span}\{e_1\}$ and $D(T) = e_1^\perp = T(0)^\perp$, which is the correct domain for a self-adjoint relation (see Remark~\ref{rem:key-observation}). The operator part is
		\[
		T_s(x_1, x_2, x_3, \dots) = (0, x_2, \tfrac{x_3}{2^2}, \tfrac{x_4}{3^2}, \dots),
		\]
		acting on $T(0)^\perp = e_1^\perp$. The operator $T_s$ is a compact diagonal operator on $e_1^\perp$ with diagonal entries $(1,\frac{1}{2^2},\frac{1}{3^2},\dots)$, so $T_s$ is compact and self-adjoint. By Lemma ~\ref{lem:compact-Ts}(ii), $T$ is compact, and since $T_s$ is self-adjoint with $D(T) = T(0)^\perp$, the relation $T$ is self-adjoint.
		
		The condition $x \in T(x)$ means there exists $c \in \mathbb{C}$ such that
		\[
		(x_1, x_2, x_3, \dots) = (c, x_2, \tfrac{x_3}{2^2}, \tfrac{x_4}{3^2}, \dots).
		\]
		Since $x \in D(T) = e_1^\perp$, we have $x_1 = 0$, hence $c = 0$. Comparing components gives
		\[
		x_n = \frac{x_n}{(n-1)^2} \quad \text{for } n \ge 3,
		\]
		which forces $x_n = 0$ for all $n \ge 3$, while $x_2$ is arbitrary. Thus
		\[
		N(I-T) = \{x \in e_1^\perp \mid x \in T(x)\} = \operatorname{span}\{(0, 1, 0, 0, \dots)\} = \operatorname{span}\{e_2\},
		\]
		so alternative (ii) holds with $Dim( N(I - T)) = 1 < \infty$.
	\end{exam}
	
	\begin{exam}\label{ex:fredholm2}
		Let $D = e_1^\perp = \{x \in \ell^2 \mid x_1 = 0\}$ and define $T \in LR(\ell^2)$ with $D(T) = D$ by
		\[
		T(x) = \{(c, 0, 0, \dots) \mid c \in \mathbb{C}\}, \qquad x \in D.
		\]
		Here $T(0) = \operatorname{span}\{e_1\}$ and $D(T) = e_1^\perp = T(0)^\perp$, consistent with self-adjointness. The operator part is $T_s = 0$ on $T(0)^\perp$, which is trivially compact, so $T$ is compact and self-adjoint.
		
		The condition $x \in T(x)$ would require $x = (c, 0, 0, \dots)$ for some $c \in \mathbb{C}$, but since $x \in D = e_1^\perp$, we have $x_1 = 0$, forcing $c = 0$ and hence $x = 0$. Thus alternative (ii) fails. For any $y \in \ell^2$, the inclusion $y \in (I-T)(x)$ is solved by $c = -y_1$ and $x = (0, y_2, y_3, \dots) \in D$, so alternative (i) holds with a unique solution.
	\end{exam}
	\subsection{Spectral theorem for compact self-adjoint linear relations}

	\begin{thm}[Spectral Theorem for compact self-adjoint linear relations]\label{thm:spectral}
		Let $H$ be a complex Hilbert space and let $T \in LR(H)$ be a compact self-adjoint linear relation. Then the following hold.
		
		\begin{enumerate}[label=(\alph*)]
			\item There exists a finite or infinite sequence $\{s_n\}_{n=1}^N$ (where $N \in \mathbb{N}_0 \cup \{\infty\}$) of nonzero real numbers such that
			\[
			\textrm{if }N\ge1,\quad |s_1| \ge |s_2| \ge \dots > 0,
			\]
			and a corresponding orthonormal sequence $\{u_n\}_{n=1}^N \subset T(0)^\perp$ such that
			\[
			(u_n, s_n u_n) \in T, \quad \text{that is, } T(u_n) = s_n u_n + T(0), \qquad n = 1, 2, \dots, N.
			\]
			If $N = \infty$, then $s_n \to 0$ as $n \to \infty$.
			
			\item The operator part $T_s$ has the diagonal representation
			\[
			T_s(x) = \sum_{n=1}^N s_n \langle x, u_n \rangle u_n
			\qquad \text{for all } x \in T(0)^\perp,
			\]
			where the series converges in the norm of $H$.
			
			\item For every $x \in D(T)$ and every $y \in T(x)$, there exists an element $y_\infty \in T(0)$ such that
			\[
			y = \sum_{n=1}^N s_n \langle x, u_n \rangle u_n + y_\infty.
			\tag{3.1}
			\]
			The series in (3.1) converges in $H$. Moreover, $\{u_n\}_{n=1}^N$ is an orthonormal basis for $\overline{R(T_s)} \subset T(0)^\perp$.
			
			\item 
			\[
			\sigma(T) \setminus \{0\} = \{s_n\}_{n=1}^N.
			\]
			Each nonzero $s_n$ occurs only finitely many times in the sequence, so every nonzero spectral value is an eigenvalue of finite multiplicity. In particular, $0$ is the only possible accumulation point of the nonzero spectrum.
			
			\item 
			The set $\{u_n\}_{n=1}^N$ can be extended by an orthonormal basis
			$\{v_\alpha\}$ of $N(T_s)=N(T)$ and an orthonormal basis
			$\{w_\beta\}$ of $T(0)$ so that
			\begin{itemize}
				\item $\{u_n\}_{n=1}^N\cup\{v_\alpha\}$ is an orthonormal basis of
				$T(0)^\perp$;
				\item $\{u_n\}_{n=1}^N\cup\{v_\alpha\}\cup\{w_\beta\}$ is an
				orthonormal basis of $H$.
			\end{itemize}
			Here $N(T_s)=N(T)$ because, for a self-adjoint relation,
			$D(T)\subset T(0)^\perp$ (Remark~\ref{rem:key-observation}).
			For $x\in D(T)=D(T_s)$, the condition $0\in T(x)$ is equivalent to
			$T_sx\in T(0)\cap T(0)^\perp=\{0\}$, and hence  $T_sx=0$.
			The vectors $u_n$ and $v_\alpha$ are eigenvectors of $T$,
			corresponding respectively to the eigenvalues $s_n$ and $0$.
			The vectors $w_\beta$ form an orthonormal basis of the multivalued
			part $T(0)$ and are not eigenvectors of $T$ in the usual sense.
		\end{enumerate}
	\end{thm}
	
	\begin{proof}
		Let $H_s = T(0)^\perp$. Since $T$ is self-adjoint, $T$ is closed. By Theorem~\ref{thm:spectrum-decomp}(iii), $T_s$ is a self-adjoint operator in $H_s$, and $T = T_s \oplus T_\infty$ with $T_\infty = \{0\} \times T(0)$. By Lemma~\ref{lem:compact-Ts}(i), $T_s$ is compact.
		
		The operator $T_s$ is allowed to be the zero operator; in that case the sequence of nonzero eigenvalues is empty, that is, $N=0$, and the conclusions are interpreted accordingly.

		Now we prove part (a) and (b).
		Apply the classical spectral theorem for compact self-adjoint operators to the operator $T_s$ on $H_s$. There exists a finite or infinite sequence $\{s_n\}_{n=1}^N$ of nonzero real numbers, where $N\in\mathbb N_0\cup\{\infty\}$, such that, when $N\ge1$, $|s_1|\ge|s_2|\ge\cdots>0$, together with a corresponding orthonormal set $\{u_n\}_{n=1}^N \subset H_s$, such that:
		\begin{enumerate}
			\item[(i)] $T_s(u_n) = s_n u_n$ for all $n$,
			\item[(ii)] for every $x \in D(T_s)$, $T_s(x) = \sum_{n=1}^N s_n \langle x, u_n \rangle u_n$,
			\item[(iii)] $\{u_n\}_{n=1}^N$ is an orthonormal basis for $\overline{R(T_s)} = N(T_s)^\perp$ in $H_s$,
			\item[(iv)] $\sigma(T_s) \setminus \{0\} = \{s_n\}_{n=1}^N$, each $s_n$ has finite multiplicity, and if $N = \infty$, then $s_n \to 0$.
			
		\end{enumerate}
		
		Since $T_s \subset T$, property (i) gives $(u_n, s_n u_n) \in T$, which means $T(u_n) = s_n u_n + T(0)$. Thus $u_n$ is an eigenvector of $T$ with eigenvalue $s_n$. This proves part (a). The orthogonality follows from the self-adjointness of $T_s$. Also property (ii) proves part (b).

To prove  part (c), consider $x \in D(T) = D(T_s)$ and $y \in T(x)$. Since $T = T_s \oplus T_\infty$, we have $y \in T_s(x) + T(0)$. Thus there exists $y_\infty \in T(0)$ with $y = T_s(x) + y_\infty$. By part (b), this gives (3.1).
		
	Proof of part (d) is as follows.
		By Theorem~\ref{thm:spectrum-decomp}(ii), $\sigma(T)=\sigma(T_s)$. Property (iv) of the classical compact self-adjoint spectral theorem gives
		\[
		\sigma(T_s)\setminus\{0\}=\{s_n\}_{n=1}^N,
		\]
		with each nonzero spectral value having finite multiplicity and with $0$ as the only possible accumulation point. Therefore
		\[
		\sigma(T)\setminus\{0\}=\{s_n\}_{n=1}^N.
		\]
		
	Now we prove part (e). From the classical spectral theorem, $\{u_n\}_{n=1}^N$ is an orthonormal basis for $\overline{R(T_s)} = N(T_s)^\perp$ in $H_s$. Let $\{v_\alpha\}$ be an orthonormal basis of $N(T_s)$. Then $\{u_n\}_{n=1}^N \cup \{v_\alpha\}$ is an orthonormal basis of $H_s = T(0)^\perp$.
		
		We verify that $N(T_s) = N(T)$. For $x \in D(T) = D(T_s)$, the condition $0 \in T(x)$ means $T_s(x) \in T(0)$ (since $T(x) = T_s(x) + T(0)$). But $T_s(x) \in T(0)^\perp$, so $T_s(x) \in T(0) \cap T(0)^\perp = \{0\}$, hence $T_s(x) = 0$. Thus $x \in N(T_s)$. Conversely, if $x \in N(T_s)$, then $(x, 0) \in T_s \subset T$, so $0 \in T(x)$, giving $x \in N(T)$. Therefore $N(T_s) = N(T)$.
		
		Each $v_\alpha \in N(T_s) = N(T)$ satisfies $(v_\alpha, 0) \in T_s \subset T$, so $v_\alpha$ is an eigenvector of $T$ with eigenvalue $0$. Each $u_n$ is an eigenvector of $T$ with eigenvalue $s_n$ as shown previously.
		
		Finally, let $\{w_\beta\}$ be an orthonormal basis of $T(0)$. Since $H = T(0)^\perp \oplus T(0)$, the union $\{u_n\} \cup \{v_\alpha\} \cup \{w_\beta\}$ is an orthonormal basis of $H$. The vectors $u_n$ and $v_\alpha$ are eigenvectors of $T$. The vectors $w_\beta$ form an orthonormal basis of the multivalued part $T(0)$ and are not eigenvectors of $T$ in the usual sense.
	\end{proof}

	\begin{rem}
		The proof shows that the spectral theorem for compact self-adjoint linear relations is a direct consequence of the Arens decomposition and the classical operator theorem. The multivalued part $T(0)$ contributes only the $y_\infty$ term in the representation (3.1) and the additional multivalued directions $\{w_\beta\}$. No finite-dimensionality condition on $T(0)$ is required anywhere in the proof; the only requirement is that $T_s$ be a compact self-adjoint operator, which follows automatically from the compactness and self-adjointness of $T$.
	\end{rem}
	
	We now present examples illustrating the spectral theorem.
	
	\begin{exam}\label{ex:spectral-shift}
		Consider the relation $F$ from Example~\ref{ex:shift}. Its adjoint is
		\[
		F^*(x_1, x_2, x_3, \dots) = (x_2, \tfrac{x_3}{2}, \tfrac{x_4}{3}, \dots), \qquad D(F^*) = e_1^\perp.
		\]
		Let $T = FF^*$. Then $T \in LR(\ell^2)$ is self-adjoint, with $D(T) = e_1^\perp$,
		\[
		T(x_1, x_2, x_3, \dots) = \{(c, x_2, \tfrac{x_3}{2^2}, \tfrac{x_4}{3^2}, \dots) \mid c \in \mathbb{C}\}.
		\]
		Here $T(0) = \operatorname{span}\{e_1\}$ and $D(T) = e_1^\perp = T(0)^\perp$. The operator part is
		\[
		T_s(x_1, x_2, x_3, \dots) = (0, x_2, \tfrac{x_3}{2^2}, \tfrac{x_4}{3^2}, \dots),
		\]
		acting on $T(0)^\perp = e_1^\perp$. The eigenvalues are $\{1, \frac{1}{2^2}, \frac{1}{3^2}, \dots\}$ with eigenvectors $e_2, e_3, e_4, \dots$. For $x \in D(T)$ and $y \in T(x)$, the spectral representation is
		\[
		y = \sum_{n=2}^\infty \frac{1}{(n-1)^2} \langle x, e_n \rangle e_n + c\,e_1,
		\]
		for a single arbitrary constant $c \in \mathbb{C}$.
	\end{exam}
	
	\begin{exam}\label{ex:infinite-T0-nontrivial}
		Let $H = \ell^2(\mathbb{N}) \oplus L^2[0,1]$. Define $T(0) = \ell^2(\mathbb{N}) \oplus \{0\}$, an infinite-dimensional closed subspace with $T(0)^\perp = \{0\} \oplus L^2[0,1] \cong L^2[0,1]$.
		
		Let $K: L^2[0,1] \to L^2[0,1]$ be the rank-one self-adjoint integral operator
		\[
		(Kf)(x) = \int_0^1 \sin(\pi x) \sin(\pi y) \, f(y) \, dy,
		\]
		with eigenvalue $\lambda = 1/2$ and normalized eigenvector $u(x) = \sqrt{2}\sin(\pi x)$.
		
		Define $T \in LR(H)$ with $D(T) = T(0)^\perp = \{0\} \oplus L^2[0,1]$ by
		\[
		T((0, f)) = T(0) + (0, Kf) = \{(b, Kf) \mid b \in \ell^2(\mathbb{N})\}, \qquad (0, f) \in D(T).
		\]
		Here $D(T) = T(0)^\perp$, consistent with the necessary condition $D(T) \subset T(0)^\perp$ for self-adjointness (Remark~\ref{rem:key-observation}). The relation $T$ is self-adjoint since $T_s = K$ is self-adjoint on $T(0)^\perp$ and $T_\infty$ is self-adjoint on $T(0)$, and their orthogonal direct sum is self-adjoint by Lemma~\ref{lem:converse-self-adjoint}.
		
		Then $T$ is self-adjoint and compact. Applying Theorem~\ref{thm:spectral}: $s_1 = 1/2$, $u_1 = (0, u)$. For $x = (0, f) \in D(T)$ and $y = (b, Kf) \in T(x)$,
		\[
		y = \tfrac{1}{2} \langle f, u \rangle_{L^2} \, (0, u) + (b, 0),
		\]
		where $(b, 0) \in T(0)$. This demonstrates that the spectral theorem applies even when $Dim ~T(0) = \infty$.
	\end{exam}
	\section{Application: Eigenfunction expansion for the Neumann Laplacian}
	
	In this section we apply the abstract spectral theorem (Theorem~\ref{thm:spectral}) to a classical elliptic boundary value problem. The Neumann Laplacian on a bounded domain has a one-dimensional kernel consisting of constant functions, so its inverse is not a single-valued operator. Using the standard quasi-boundary-triple framework to identify the Neumann realization, we then identify its inverse as a compact self-adjoint linear relation and obtain the well-known eigenfunction expansion with an arbitrary additive constant.
	
	\subsection{Quasi-boundary triples and self-adjoint extensions}
	
	We briefly recall the framework needed for the Laplacian, following Derkach and Malamud~\cite{DerkachMalamud}. For elliptic differential operators on smooth bounded domains, the natural trace maps are generally not an ordinary boundary triple on the full maximal domain. A convenient rigorous framework is provided by a \emph{quasi-boundary triple} introduced in \cite{quasiB}. In the present application, the boundary maps are defined on a suitable intermediate domain containing the minimal domain, and the associated Green identity and extension parametrisation are used in the standard way.
	
	Let $A_0$ be a densely defined closed symmetric operator in a Hilbert space $H$, and let $T$ be an operator satisfying
	\[
	A_0\subset T\subset A_0^*.
	\]
	A triple $(\mathcal K,\Gamma_0,\Gamma_1)$ is a quasi-boundary triple for $T$ if $\mathcal K$ is a Hilbert space, $\Gamma_0,\Gamma_1:{D}(T)\to\mathcal K$ are linear maps such that
	\begin{enumerate}[label=(\roman*)]
		\item the map
		\[
		{D}(T)\ni x\longmapsto(\Gamma_0x,\Gamma_1x)\in\mathcal K\times\mathcal K
		\]
		has dense range;
		\item the abstract Green identity
		\[
		\langle Tx,y\rangle-\langle x,Ty\rangle
		=\langle\Gamma_0x,\Gamma_1y\rangle_{\mathcal K}
		-\langle\Gamma_1x,\Gamma_0y\rangle_{\mathcal K}
		\]
		holds for all $x,y\in D(T)$;
		\item the operator
		\[
		A_{\mathrm{ref}}:=T|_{\ker\Gamma_0}
		\]
		is self-adjoint in $H$.
	\end{enumerate}
	
	For the Laplacian on a smooth bounded domain, one may take the intermediate domain to be $H^2(\Omega)$ and use the trace maps
	\[
	\Gamma_0u=u|_{\partial\Omega},\qquad
	\Gamma_1u=\partial_\nu u|_{\partial\Omega},
	\]
	with the usual Sobolev-space interpretation. Green's second identity gives the required abstract Green identity. The Neumann realization corresponds to
	\[
	A_N=-\Delta\big|_{\ker\Gamma_1},
	\]
	with domain
	\[
	D(A_N)=\{u\in H^2(\Omega):\partial_\nu u=0\text{ on }\partial\Omega\}.
	\]
	Thus the only role of the boundary-triple framework in this paper is to identify the Neumann realization as a self-adjoint extension; the subsequent compactness and spectral arguments are carried out directly for this operator and its inverse relation.
	
	\subsection{The Neumann Laplacian and its inverse as a linear relation}
	
	Let $\Omega \subset \mathbb{R}^n$ be a bounded domain with smooth boundary $\partial\Omega$. Consider the minimal symmetric operator associated with the Laplacian:
	\begin{equation}\label{eq:A0}
		A_0 u = -\Delta u, \qquad
		D(A_0) = H^2_0(\Omega) = \{u \in H^2(\Omega) \mid u|_{\partial\Omega} = 0 \text{ and } \partial_\nu u|_{\partial\Omega} = 0\},
	\end{equation}
	where $\partial_\nu$ denotes the outward normal derivative. The space $H^2_0(\Omega)$ is the closure of $C_c^\infty(\Omega)$ in $H^2(\Omega)$; equivalently, it consists of $H^2$ functions with vanishing zeroth and first order traces on $\partial\Omega$. $A_0$ is a closed symmetric operator with infinite deficiency indices. Its adjoint is the maximal operator
    \begin{center}
         $	A_0^* u = -\Delta u,$
    \end{center}
	\begin{equation}\label{eq:A0star}
D(A_0^*) = \{u \in L^2(\Omega) \mid \Delta u \in L^2(\Omega) \text{ in the sense of distributions}\}.
	\end{equation}

	The Neumann realization identified above is the self-adjoint operator
	\begin{equation}\label{eq:Neumann}
		A=A_N=-\Delta,\qquad
		D(A)=\{u\in H^2(\Omega):\partial_\nu u=0\text{ on }\partial\Omega\}.
	\end{equation}
	Because $\Omega$ is bounded and has smooth boundary, $A$ has compact resolvent: $(A + i)^{-1}$ is a compact operator on $L^2(\Omega)$. This follows from the Rellich--Kondrachov theorem, which gives the compactness of the embedding $H^2(\Omega) \hookrightarrow L^2(\Omega)$, together with the  estimate $\|u\|_{H^2(\Omega)} \le C \|(A+i)u\|_{L^2(\Omega)}$ for $u \in D(A)$. Consequently, the spectrum of $A$ consists of a sequence of eigenvalues
	\[
	0 = \lambda_0 < \lambda_1 \le \lambda_2 \le \dots \to \infty,
	\]
	each of finite multiplicity. The eigenspace for $\lambda_0 = 0$ is the space of constant functions,
	\[
	N(A) = \operatorname{span}\{\mathbf{1}\},
	\]
	where $\mathbf{1}$ denotes the constant function with value $1$. Indeed, if $-\Delta u = 0$ with $\partial_\nu u = 0$, then Green's identity gives
	\[
	\int_\Omega |\nabla u|^2 \, dx = \langle -\Delta u,u\rangle_{L^2(\Omega)} + \int_{\partial\Omega} \partial_\nu u \, \overline{u} \, dS = 0,
	\]
	so $\nabla u = 0$ and $u$ is constant.
	
	Since $0$ is an eigenvalue, $A$ is not invertible as an operator. However, its inverse can be defined as a linear relation. Set
	\begin{equation}\label{eq:inverse}
		T = A^{-1} = \bigl\{ (f, u) \in L^2(\Omega) \times L^2(\Omega) \mid u \in D(A),\; -\Delta u = f,\; \partial_\nu u = 0 \text{ on } \partial\Omega \bigr\}.
	\end{equation}
	This means $(f, u) \in T$ if and only if $(u, f) \in A$, that is, $u$ is a solution of the Neumann problem with right-hand side $f$.
	
	For $f \in R(A) = N(A)^\perp = \{f \in L^2(\Omega) \mid \int_\Omega f \, dx = 0\}$, the equation $-\Delta u = f$ with $\partial_\nu u = 0$ has a unique solution $u_0$ orthogonal to the constants; the full set of solutions is $u_0 + \operatorname{span}\{\mathbf{1}\}$. The uniqueness of $u_0$ follows from the fact that the difference of two solutions is a harmonic function with $\partial_\nu u = 0$, hence a constant; requiring orthogonality to constants gives uniqueness. Existence follows from the Lax--Milgram theorem applied to the weak formulation on the Hilbert space of zero-mean $H^1$ functions. Thus $T$ is a linear relation with
	\[
	D(T) = R(A) = \Bigl\{f \in L^2(\Omega) \mid \int_\Omega f \, dx = 0\Bigr\}, \qquad
	T(0) = N(A) = \operatorname{span}\{\mathbf{1}\}.
	\]
	Note that $D(T) = R(A) = N(A)^\perp = T(0)^\perp$, which is consistent with the general inclusion $D(T) \subseteq T(0)^\perp$ for self-adjoint linear relations (Remark~\ref{rem:key-observation}).
	
	\begin{lem}\label{lem:properties}
		The linear relation $T = A^{-1}$ is self-adjoint, compact, and $Dim T(0) = 1$.
	\end{lem}
	
	\begin{proof}
The inverse of a self-adjoint linear relation is self-adjoint \cite{21}. Explicitly, $T^* = (A^{-1})^* = (A^*)^{-1} = A^{-1} = T$, and hence the relation $T = A^{-1}$ is self-adjoint.
		
To prove the compactness of a linear relation $T$, we have to show the compactness of the induced operator $QT$, where $Q: L^2(\Omega) \to L^2(\Omega)/\overline{T(0)}$ is the quotient map. Since $T(0) = \operatorname{span}\{\mathbf 1\}$ is already closed, $\overline{T(0)} = T(0)$. Let $(f_n)$ be a bounded sequence in $D(T) = R(A)$. For each $n$, choose $u_n \in T(f_n)$; then $-\Delta u_n = f_n$ and $\partial_\nu u_n = 0$. Decompose $u_n = u_n^0 + c_n \mathbf{1}$, where
		\[
		u_n^0 = u_n - \frac{1}{|\Omega|} \int_\Omega u_n \, dx
		\]
		has zero mean and $c_n = \frac{1}{|\Omega|} \int_\Omega u_n \, dx$. Since $Q$ annihilates $\operatorname{span}\{\mathbf 1\}$, we have $QT(f_n) = Q(u_n) = Q(u_n^0)$, so it suffices to produce an $L^2$-convergent subsequence of $(u_n^0)$; the constants $c_n$ play no role in the quotient.
		
		We show $(u_n^0)$ is bounded in $H^1(\Omega)$. Since $f_n \in R(A) = N(A)^\perp$, we have $\int_\Omega f_n \, dx = 0$, so $\langle f_n, \mathbf{1} \rangle = 0$. Using Green's identity and the Neumann boundary condition:
		\[
		\int_\Omega |\nabla u_n|^2 \, dx = \int_\Omega (-\Delta u_n) u_n \, dx = \langle f_n, u_n \rangle = \langle f_n, u_n^0 \rangle + c_n \langle f_n, \mathbf{1} \rangle = \langle f_n, u_n^0 \rangle.
		\]
		By the Cauchy--Schwarz inequality and the Poincar\'e inequality for zero-mean functions ($\|u_n^0\|_{L^2} \le C_P \|\nabla u_n^0\|_{L^2} = C_P \|\nabla u_n\|_{L^2}$), we get
		\[
		\|\nabla u_n\|_{L^2}^2 \le \|f_n\|_{L^2} \|u_n^0\|_{L^2} \le C_P \|f_n\|_{L^2} \|\nabla u_n\|_{L^2}.
		\]
		Since $(f_n)$ is bounded, this gives $\|\nabla u_n\|_{L^2} \le C_P \sup_n \|f_n\|_{L^2} < \infty$. By the Poincar\'e inequality again,
		\[
		\|u_n^0\|_{H^1(\Omega)}^2 = \|u_n^0\|_{L^2}^2 + \|\nabla u_n\|_{L^2}^2 \le (1+C_P^2)\|\nabla u_n\|_{L^2}^2 \le C,
		\]
		so $(u_n^0)$ is bounded in $H^1(\Omega)$.
		
		By the Rellich--Kondrachov theorem, the embedding $H^1(\Omega) \hookrightarrow L^2(\Omega)$ is compact, so there exists a subsequence $(u_{n_k}^0)$ converging in $L^2(\Omega)$ to some $u_*^0 \in L^2(\Omega)$. Consequently $Q(u_{n_k}^0) \to Q(u_*^0)$ in $L^2(\Omega)/T(0)$, that is $QT(f_{n_k}) = QT$ applied to the subsequence converges. Thus $QT$ is compact, and $T$ is compact.
		
		\emph{Dimension:} Clearly $Dim~ T(0) = Dim~ N(A) = 1$, since $N(A) = \operatorname{span}\{\mathbf{1}\}$.
	\end{proof}
	
	Thus $T$ satisfies all hypotheses of Theorem~\ref{thm:spectral}: $T$ is compact, self-adjoint, $Dim~ T(0) = 1$, and $D(T) \neq \{0\}$ (since $R(A)$ contains nonzero functions).
	
	\subsection{Eigenfunction expansion for the Neumann problem}
	
	Applying Theorem~\ref{thm:spectral} to $T = A^{-1}$, we obtain a sequence of nonzero real numbers $\{s_k\}_{k=1}^\infty$ with $|s_1| \ge |s_2| \ge \cdots > 0$, $s_k \to 0$, and an orthonormal set $\{u_k\}_{k=1}^\infty \subset T(0)^\perp \subset L^2(\Omega)$ such that
	\[
	(u_k, s_k u_k) \in T, \quad \text{that is, } T(u_k) = s_k u_k + T(0) = s_k u_k + \operatorname{span}\{\mathbf{1}\}, \qquad k = 1, 2, \dots
	\]
	By Theorem~\ref{thm:spectral}(e), the functions $u_k$ are orthogonal to $N(T) = T(0) = \operatorname{span}\{\mathbf{1}\}$, so $\int_\Omega u_k = 0$.
	
	We now identify the $u_k$ as the Neumann eigenfunctions. Since $u_k \in D(T) = R(A)$, we have $\int_\Omega u_k = 0$ automatically. The condition $T(u_k) = s_k u_k + \operatorname{span}\{\mathbf{1}\}$ means that for any $v \in T(u_k)$, there exists a constant $c \in \mathbb{C}$ such that $v = s_k u_k + c \mathbf{1}$. In particular, $s_k u_k \in T(u_k)$. By definition of $T = A^{-1}$, this means $(s_k u_k, u_k) \in A$, that is,
	\[
	s_k u_k \in D(A), \qquad -\Delta(s_k u_k) = u_k, \qquad \partial_\nu(s_k u_k) = 0.
	\]
	Since $s_k \neq 0$, this gives
	\[
	u_k \in D(A), \qquad -s_k \Delta u_k = u_k, \qquad \partial_\nu u_k = 0.
	\]
	Dividing by $s_k$, we obtain the eigenvalue equation
	\[
	-\Delta u_k = \lambda_k u_k, \qquad \lambda_k := s_k^{-1} > 0, \qquad \partial_\nu u_k = 0.
	\]
	Note that no additional constant term appears because $\Delta \mathbf{1} = 0$ and $\partial_\nu \mathbf{1} = 0$; the multivalued part of $T$ affects the choice of representative in $T(u_k)$ but not the eigenvalue equation itself.
	
	Thus $\{u_k\}_{k=1}^\infty$ are precisely the non-constant Neumann eigenfunctions of $A$ (normalized in $L^2(\Omega)$), and $\{\lambda_k\}_{k=1}^\infty$ are the corresponding positive eigenvalues. The spectral theorem guarantees that the $u_k$ form an orthonormal set, are orthogonal to the constants, and that every positive Neumann eigenvalue appears in this sequence with the correct multiplicity. Since the Neumann spectrum consists exactly of $0$ and the positive eigenvalues, and the spectral theorem gives $\sigma(T) \setminus \{0\} = \{s_k\}_{k=1}^\infty$, we recover the full spectral structure.
	
	Now take any $f \in D(T) = R(A)$ (that is, $\int_\Omega f \, dx = 0$) and any solution $u \in T(f)$ of the Neumann problem
	\[
	-\Delta u = f \quad \text{in } \Omega, \qquad \partial_\nu u = 0 \quad \text{on } \partial\Omega.
	\]
	By the expansion formula in Theorem~\ref{thm:spectral}(c), there exists $w \in T(0) = \operatorname{span}\{\mathbf{1}\}$ such that
	\[
	u = \sum_{k=1}^\infty s_k \langle f, u_k \rangle u_k + w.
	\]
	Substituting $s_k = \lambda_k^{-1}$ and writing $w = C \mathbf{1}$ for some constant $C \in \mathbb{C}$ yields the classical eigenfunction expansion
	\begin{equation}\label{eq:eigen-expansion}
		u(x) = \sum_{k=1}^\infty \frac{\langle f, u_k \rangle}{\lambda_k} \, u_k(x) \;+\; C,
	\end{equation}
	where $C$ is an arbitrary constant. The series converges in $L^2(\Omega)$ by Theorem~\ref{thm:spectral}. Moreover, the partial sums converge in the graph norm of the Neumann operator $A$: applying $A$ to the partial sums gives the corresponding Fourier partial sums of $f$, which converge in $L^2(\Omega)$. By the elliptic estimate for the Neumann realization, the graph norm controls the $H^2(\Omega)$-norm, so the series also converges in $H^2(\Omega)$ under the stated smoothness assumptions.
	
	This is exactly the expansion provided by the classical spectral theorem for the compact self-adjoint operator $T_s$ (the inverse of $A$ restricted to the orthogonal complement of the constants), but here it is obtained directly from the abstract result for compact self-adjoint linear relations, without first projecting onto the orthogonal complement of the constants. The multivaluedness of $T$ elegantly captures the arbitrary additive constant in the solution, which is the hallmark of the Neumann problem: solutions are only determined up to an additive constant.
	
	\begin{rem}\label{rem:neumann-expansion-sign}
		The eigenvalues $\lambda_k$ defined here satisfy $-\Delta u_k = \lambda_k u_k$ with $\lambda_k > 0$. The classical Neumann eigenvalues are nonnegative, with $\lambda_0 = 0$ corresponding to the constant eigenfunction. Our $\lambda_k$ for $k \ge 1$ are precisely the positive Neumann eigenvalues.
	\end{rem}
	
	\begin{rem}\label{rem:generalization}
		The same technique applies to any self-adjoint elliptic operator with compact resolvent that has a finite-dimensional kernel. For example, the Robin problem
		\[
		\partial_\nu u + \alpha u = 0 \quad \text{on } \partial\Omega, \qquad \alpha \in \mathbb{R},
		\]
		has a non-trivial kernel precisely when $-\alpha$ is an eigenvalue of the associated Steklov problem
		\[
		\Delta v = 0 \text{ in } \Omega, \qquad \partial_\nu v = \sigma v \text{ on } \partial\Omega;
		\]
		since the Steklov eigenvalues $\sigma$ satisfy $\sigma \ge 0$ (with $\sigma = 0$ corresponding to the constants), a non-trivial kernel of the Robin realization can only occur for $\alpha \le 0$, and in that case $Dim~ N(A)$ equals the multiplicity of $-\alpha$ as a Steklov eigenvalue. More generally, the theory of boundary triples allows one to prescribe boundary conditions that are themselves linear relations (set-valued), leading to multi-valued self-adjoint realisations of elliptic operators. The inverses of such realisations are again compact self-adjoint linear relations, and under the same type of compactness argument as in Lemma~\ref{lem:properties}, Theorem~\ref{thm:spectral} applies to give a unified framework for eigenfunction expansions in elliptic boundary value problems with set-valued boundary conditions.
	\end{rem}

\section{Concluding Remarks}

We have proved a spectral theorem for compact self-adjoint linear relations
on a Hilbert space without imposing any finite-dimensionality condition on the
multivalued part \(T(0)\). The proof is a reduction: the Arens decomposition
\(T=T_s\oplus T_\infty\) separates the single-valued part \(T_s\), a compact
self-adjoint operator on \(T(0)^\perp\), to which the classical spectral
theorem applies, while \(T(0)\) contributes only the additive ambiguity in
the representation
\(y=\sum_n s_n\langle x,u_n\rangle u_n + y_\infty\).
The novelty here is not the spectral theorem itself, which is classical, but
the observation that it extends verbatim to compact self-adjoint linear
relations once the Arens decomposition is applied, and in particular that the
finite-dimensionality condition on \(T(0)\), often imposed in the literature,
is unnecessary.

The key technical point is that for a self-adjoint \(T\), the domain satisfies
\(D(T)\subset T(0)^\perp\), so the operator part \(T_s\) carries all the
single-valued information of \(T\) and the multivalued part \(T_\infty =
\{0\}\times T(0)\) contributes only the additive ambiguity.

The same reduction yields a finite-rank approximation theorem for compact
bounded closed linear relations, and a Fredholm alternative for compact
self-adjoint linear relations; in both cases the multivalued part plays no
obstruction. The examples show that the Fredholm alternative genuinely requires
self-adjointness: for compact non-self-adjoint linear relations, the
dimensions of \(N(I-T)\) and \(N(I-T^*)\) may differ, so the operator-theoretic
statement does not carry over without modification.

The application to the Neumann Laplacian is an illustration rather than a new
PDE result. The classical eigenfunction expansion is recovered by observing
that the inverse of the Neumann Laplacian is a compact self-adjoint linear
relation whose multivalued part is the span of the constants, and the abstract
theorem then produces the expansion with the arbitrary additive constant
directly. The quasi-boundary-triple framework is used only to identify the
Neumann realization as a self-adjoint extension of the minimal Laplacian.

Natural directions for further work are the following.
\begin{itemize}
	\item Extending the spectral theorem to normal compact linear relations,
	where a suitable operator part must be defined, possibly via a polar
	decomposition in the multivalued setting.
	\item Developing perturbation and approximation results for linear relations,
	building on the finite-rank approximation theorem proved here.
	\item Applying the framework to elliptic operators on manifolds or to
	higher-order operators with finite-dimensional kernel.
	\item Treating unbounded self-adjoint linear relations and their functional
	calculus.
	\item Studying relations that arise from set-valued boundary conditions, which
	appear naturally in problems with impedance or homogenised boundary data.
\end{itemize}
	\subsection*{Acknowledgment}
	The authors wish to thank Guixin Xu, Bisna Mary Eldo, and Tathagata Banerjee for their valuable suggestions and comments.
	\subsection*{Disclosure statement}
	No potential conflict of interest was reported by the authors.
\printbibliography
\end{document}